\documentclass[twoside,a4paper,12pt]{amsart}
\usepackage{amsfonts}
\usepackage{amssymb}
\usepackage{amsfonts}
\usepackage{amssymb}
\usepackage{amsthm}
\usepackage{amsmath}
\usepackage{a4wide}
\usepackage{amsthm}
\usepackage[usenames,dvipsnames,svgnames,table]{xcolor}

\theoremstyle{plain}
\newtheorem{teo}{Theorem}[section]
\newtheorem{Lemma}[teo]{Lemma}
\newtheorem{Proposition}[teo]{Proposition}

\newtheorem{ackn}{Acknowledgement\!}

\theoremstyle{definition}
\newtheorem{Definition}[teo]{Definition}

\theoremstyle{remark}
\newtheorem{Remark}[teo]{Remark}

\def\R{\mathbb R}

\newcommand{\D }{\Delta }

\newcommand{\e }{\varepsilon }

\renewcommand{\i }{\iota}

\newcommand{\hF}{\hat{F}}

\newcommand{\s }{\sigma }
\newcommand{\Sig }{\Sigma}
\newcommand{\hSig }{\hat{\Sigma}}
\renewcommand{\t }{\tau }

\newcommand{\intbar}{\etaathop{\int\etaakebox(-13.5,0){\rule[4pt]{.7em}{0.3pt}}%
\kern-6pt}\nolimits}

\newcommand{\be}{\begin{equation}}
\newcommand{\ee}{\end{equation}}
\newcommand{\bea}{\begin{equation*}}
\newcommand{\eea}{\end{equation*}}

\newcommand{\bg}{\bar{g}}

\newcommand{\partialeo}{\partial_{\epsilon}{\Big|}_{\epsilon=0}}

\newcommand{\grad}{\nabla}

\newcommand{\bgrad}{\bar{\nabla}}

\newcommand{\mE}{\mathcal{E}}

\newcommand{\mW}{\mathcal{W}}

\newcommand{\hf}{\hat{f}}

\theoremstyle{plain}

\newtheorem{lemma}[teo]{Lemma}
\newtheorem{prop}[teo]{Proposition}

\theoremstyle{definition}
\newtheorem{dfnz}[teo]{Definition}

\newtheorem{lem}{Lemma}[section]

\newtheorem{thm}[lem]{Theorem}

\theoremstyle{remark}

\def\R{{{\mathbb R}}}
\def\SS{{{\mathbb S}}}

\def\RRR{{\mathrm R}}

\def\AAA{{\mathrm A}}

\def\HHH{{\mathrm H}}

\def\RRR{{\mathrm R}}

\def\hF{\hat{F}}

\newcommand{\N}{\mathbb{N}}

\def\be{\begin{equation}}
\def\ee{\end{equation}}
\def\bea{\begin{eqnarray*}}
\def\bean{\begin{eqnarray}}
\def\eean{\end{eqnarray}}
\def\eea{\end{eqnarray*}}

\author[A. Magni]
{Annibale Magni}
\address[Annibale Magni]{Institut f\"ur Angewandte Mathematik, Westf\"alische--Wilhelms Universit\"at M\"unster, Einsteinstrasse 62, 48149 M\"unster (Germany).}
\email[]{magni@uni-muenster.de}



\title
[A convergence result for the Gradient Flow of $\int |\AAA|^2$ in Riemannian Manifolds]
{A convergence result for the Gradient Flow of $\int |\AAA|^2$ in Riemannian Manifolds}

\begin{document}

\maketitle

\begin{abstract}
We study the gradient flow of the $L^2-$norm of the second fundamental form for smooth immersions of two-dimensional surfaces into compact Riemannian manifolds. By analogy with the results obtained in \cite{Link13} and \cite{MeWW13} for the Willmore flow, we prove lifespan estimates in terms of the $L^2-$concentration of the second fundamental form of the initial data and we show the existence of blowup limits. Under special condition both on the initial data and on the target manifold, we prove a long time existence result for the flow and the subconvergence to a critical immersion.
\end{abstract}

\section{Introduction}
Let $(N^n, \bg)$ be a closed (i.e. compact and without boundary) $n$-dimensional Riemannian manifold and $\Sigma$ a closed surface. For  an immersion $F: \Sigma \to N^n$, with associated pullback metric $g := F^* \bg$, second fundamental form $\AAA^F$ and induced measure $\mu_F$, we consider the functional
\be \label{Energy}
\mE(F):= \int_{\Sigma} |\AAA^F|^2 d \mu_F \,.
\ee
In \cite{KuMS14}, adapting L. Simon's techniques from \cite{Simo93}, the problem of finding minimizers of $\mE$ has been addressed in the class of smooth immersions of spheres $F : \SS^2 \to N^3$ into three--dimensional Riemannian Manifolds. In particular, under conditions on the curvature of $N^3$, guaranteeing both a uniform area bound along a minimizing sequence and an upper bound for the infimum of $\mE$, the existence and the smoothness of the minimizers of $\mE$ has been proven. In \cite{MoRi12} the same problem has been addressed in arbitrary codimension and in the more general class of weak branched conformal immersions  $F : \SS^2 \to N^n$. In this setting, the authors have shown that a minimizing sequence (modulo subsequences) either shrinks to a point, or converges (in the sense of currents) to a Lipschitz immersion of $\SS^2$, whose image is made of a connected union of finitely many, possibly branched, weak immersions of $\SS^2$ with finite total curvature. In \cite{MoRi13}, smooth regularity of the minimizing immersions away from the (at most finite) branching points has been proven.\\
In the present work we study the $L^2-$gradient flow of the functional $\mE$ in $n-$dimensional Riemannian manifolds. More precisely, given an immersion $f_0 : \Sigma \to N^n$, we consider the one parameter family of immersions $f : \Sigma \times [0,T) \to N^n$ which solves the initial value problem
\begin{equation} \label{Flow}
\partial_t f = - \grad \mE (f) \,, \quad \textrm{with} \quad f(\cdot,0)=f_0(\cdot)
\end{equation}
(see Definition \ref{dfnz:A2speed} for the explicit expression of $\grad \mE$). Any solution to \eqref{Flow} (which, by the ellipticity of the operator $\grad \mE$, exists unique  for small times for any given initial datum $f_0$) will be called a $\grad \mE-$flow.\\
In the first part of the paper we obtain a priori estimates on the life span of a $\grad \mE-$flow in terms of the concentration of the $L^2$-norm of the second fundamental form $\AAA^{f_0}$ at the initial time. The proofs closely follow the line adopted for the analysis of the gradient flow of the Willmore functional, which have been first addressed in \cite{KuSc01}, \cite{KuSc02} and \cite{Simo01} for immersions in Euclidean target spaces and later in \cite{Link13} and \cite{MeWW13} for immersions into Riemannian manifolds.\\
In the second part of the paper we present a long time existence result for $\grad \mE-$flows in three dimensional Riemannian manifolds. For an initial datum $f_0: \SS^2 \to N^3$ satisfying $\mE(f_0) \leq 8 \pi$ and under suitable conditions on the ambient manifold $(N^3, \bg)$, we show that the $\grad \mE-$flow starting at $f_0$ exists for all positive times and (modulo subsequences) converges to a surface which is critical for the functional $\mE$.
%
%
%
%
\section{Notation and Preliminaries}
In this section we introduce the notations and the conventions which will be used in the rest of the paper.\\
With $(N^n, \bg)$ we will denote an $n$-dimensional Riemannian manifold and with $\bgrad$ its Levi-Civita connection on $TN^n$, with associated Riemann tensor $\bar{\RRR}$.\\
$\Sigma$ will be a  two dimensional connected manifold. For an immersion $F : \Sigma \rightarrow N^n$,  we will call $g := F^* \bg$ the pullback metric on $\Sigma$ and $\mu_F$ its associated Riemannian measure. We also define $\Sigma_{\s}(x_0) := F^{-1}(B^{N^n}_{r}(x_0))$, where $B^{N^n}_{r}(x_0)$ is the geodesic ball in $N^n$ with center $x_0 \in N^n$ and radius $r$. In the case $N^n = \R^n$ with the standard Euclidean metric, we will use the convention $B^{N^n}_{r}(0) =: B^n_r$.\\
With $\perp$ we will denote the projection on the orthogonal complement of $F_*(T\Sigma) \subset TN^n$ along the immersion $F$, and $\grad$ will denote the normal connection on $(F_*(T\Sigma))^{\perp}$.\\
With $\AAA^F$ we will denote the second fundamental form of $F(\Sigma)$ in $N^n$. More explicitly, in a local basis on $\Sigma$, it holds
$$
\AAA^F_{ij} := (\bgrad_i \partial_j F)^{\perp}\,.
$$
We define the mean curvature vector of the immersion $F : \Sigma \to N^n$ according to
$$
\HHH^F := g^{ij}\AAA^F_{ij}\,.
$$
With $k_g$ we will denote the Gaussian curvature of $\Sigma$ with respect to the metric $g$ and with $\chi(\Sigma)$ its Euler characteristic.\\
We define the Willmore functional of an immersion $F: \Sigma \to N^n$ as
\begin{equation*}
\mW(F) := \frac{1}{4}\int_{\Sigma} |\HHH^F|^2 d\mu_{F}\,.
\end{equation*}

In the computation of the evolution equations of the relevant geometric quantities we will make use of the Codazzi equation
\be \label{Codazzi}
\grad_X \AAA^F(Y,Z) - \grad_Y \AAA^F(X,Z) = (\bar{\RRR}(\tilde{X},\tilde{Y},\tilde{Z}))^{\perp}\,,
\ee
where $X,Y,Z \in T\Sigma$ and $\tilde{X} = F_*X$, $\tilde{Y} = F_*Y$, $\tilde{Z} = F_*Z$.\\
The Ricci equation will be used in the following form
\be \label{Ricci}
g^{ij}(\AAA^F(\partial_i,X)\bg(\AAA^F(\partial_j,Y),V) - \AAA^F(\partial_i,Y)\bg(\AAA^F(\partial_j,X),V))=\RRR^{\perp}(X,Y) V - (\bar{\RRR}(\tilde{X},\tilde{Y})V)^{\perp} \,,
\ee
where $V$ is a normal vector field along $F$ and $\RRR^{\perp}$ is the Riemann tensor of the normal connection associated to the immersion $F$ itself.\\

We now introduce polynomial functions of the second fundamental form of $F$ and of the Riemann tensor of $\bgrad$ which will be useful to detect the structure of the evolution equations of the second fundamental form an of its covariant derivatives along a $\grad\mE-$flow.
\begin{dfnz} \label{PQnotation}
Given an immersion $F:\Sigma \to N^n$ of a smooth surface $\Sigma$ into a Riemannian manifold, we will denote with $P^k_l(\AAA^F)$ any universal linear combination of terms of the form
$$
\grad^{i_1}\AAA^F * \cdots *\grad^{i_l}\AAA^F \quad \textrm{with} \quad |i|:= i_1 + \cdots + i_l = k \,,
$$
where $*$ is a contraction.
By $Q^{k,l}_{(m)}(\AAA^F,\bar{\RRR})$ we will denote any universal linear combination of terms having the structure
$$
\bgrad^r\bar{\RRR} \circ F  * \grad^{i_1}\AAA^F * \cdots * \grad^{i_{\nu}}\AAA^F\ * \i_{\Sigma} * \cdots *\i_{\Sigma} * DF * \cdots * DF \,,
$$
where $r + |i| + \nu = k + l$, $|i| \leq k$, $r \leq m$ (in case $m$ is given), and $\i_{\Sigma}: (F_*(T\Sigma)^{\perp},\grad) \to (TN^n,\bg)$ is the canonical injection.\\
With $Q^{k,l}_{R*R}(\AAA^F,\bar{\RRR})$ we will denote universal linear combinations of terms of the form
$$
\bgrad^{r_1}\bar{\RRR} \circ F * \bgrad^{r_2}\bar{\RRR} \circ F  * \grad^{i_1}\AAA^F* \cdots * \grad^{i_{\nu}}\AAA^F * \i_{\Sigma} * \cdots *\i_{\Sigma} * DF * \cdots * DF\,,
$$
where $r_1 + r_2 + |i| + \nu = k + l$, and $|i| \leq k$.
\end{dfnz}
\begin{Remark}
In the rest of the paper we will often omit the arguments of the $P$ and $Q$ polynomials.\\
Within the notation introduced in Definition \ref{PQnotation}, the following rules hold
\be \label{rule1}
\grad P^k_l = P^{k+1}_l \,,
\ee
\be \label{rule2}
\grad Q^{k,l}_{(m)} = Q^{k+1,l}_{(m+1)}
\ee
and
\be \label{rule3}
Q^{k,l} * \AAA^F = Q^{k,l+1}\,.
\ee
\end{Remark}
%
%
\section{The Euler--Lagrange Equation for $\int_{\Sigma} |\AAA^F|^2 d \mu_F$}
In this section we derive the Euler--Lagrange equation for the functional $\mE$. By means of the $P$ and $Q$ polynomials introduced in Definition \ref{PQnotation}, we also describe the structure of the evolution equations along a $\grad\mE-$flow of the second fundamental form and its covariant derivatives.

\begin{prop}
Let $I \subset \R$ be an interval with $0 \in I$ and $f: \Sigma \times I \to N^n$ a smooth one parameter family of immersions such that $V (x, \e):= \partial_{\e} f(x,\e)$ is normal along $f$ at $\e = 0$. Then it holds

\be \label{varmass}
\partialeo d \mu_f = - \bg(V , \HHH^f) d \mu_f \,,
\ee

\be \label{var2ff}
\partialeo \AAA^f_{ij} = \grad^2_{ij} V - g^{kl} \AAA^f_{il} \bg (V , \AAA^f_{jk}) + (\bar{\RRR}(V, \partial_i f) \partial_j f)^{\perp} \,,
\ee
and
\be \label{varen}
\frac{d}{d \e}{\Big|}_{\e=0} \mE (f)= 2 \int_{\Sigma} \bg \Big( \big(g^{ip} g^{jq}\grad^2_{pq} \AAA^f_{ij} - g^{ij} g^{kl} g^{pq} \bg(\AAA^f_{ip}, \AAA^f_{lq}) \AAA^f_{jk} + \bar{\RRR}(\AAA^f_{ij},\partial_j f) \partial _i f- \frac{1}{2} |\AAA^f|^2 \HHH^f\big), V \Big)  d\mu_f \,.
\ee
\end{prop}

\begin{proof}
Equations \eqref{varmass} and \eqref{var2ff} are deduced by standard computations. \eqref{varen} follows from \eqref{varmass}, \eqref{var2ff} and \eqref{Ricci}, taking into account that $\bg(\partial_l f (x , 0), V (x,0)) = 0$ for $l \in \{ 1,2 \}$ and $x \in \Sigma$.
\end{proof}

\begin{dfnz} \label{dfnz:A2speed}
For an immersion $F: \Sigma \to N^n$ we define
\be \label{A2speed}
\grad \mE (F) =  g^{ip} g^{jq}\grad^2_{pq} \AAA^F_{ij} - g^{ij} g^{kl} g^{pq} \bg(\AAA^F_{ip} , \AAA^F_{lq}) \AAA^F_{jk} + \bar{\RRR}(\AAA^F_{ij},\partial_j F) \partial _i F- \frac{1}{2} |\AAA^F|^2 \HHH^F \,.
\ee
\end{dfnz}

\begin{lemma}
For all $a,b,p,q \in \{ 1 , 2 \}$ it holds
\be \label{com2A}
\grad^2_{ab} \AAA^F_{pq} = \grad^2_{pq} \AAA^F_{ab} + P^0_3 + Q^{0,1}_1
\ee
and
\be \label{com4A}
\grad^2_{ab} \grad^2_{pq} \AAA^F_{pq} = \grad^2_{pq} \grad^2_{ab} \AAA^F_{pq} + Q^{2,1}_1 \,.
\ee
\end{lemma}

\begin{proof}
We start by noticing that \eqref{Codazzi} can be written in the form
\be \label{Codazzi-Q}
\grad_b \AAA^F_{pq} - \grad_p \AAA^F_{bq} = (\bar{\RRR}_{bpq})^{\perp} = Q^{0,0}
\ee
and \eqref{Ricci} give
\be
\begin{split}
\grad^2_{ab} \AAA^F_{pq}  & \overset{\eqref{Codazzi-Q}} = \grad_a (\grad_p \AAA^F_{bq} + Q^{0,0})\overset{\eqref{rule2}} =  \grad^2_{ap} \AAA^F_{qb} + Q^{0,1}_1\\
                          & \overset{\eqref{Ricci}} = \grad^2_{pa}\AAA^F_{qb} + P^0_3 + Q^{0,0} * \AAA + Q^{0,1}_1  \overset{\eqref{rule3} - \eqref{Codazzi-Q}} = \grad^2_{pq} \AAA^F_{ab} + P^0_3 + Q^{0,1}_1\\
\end{split}
\ee
which is \eqref{com2A}. Equation \eqref{com4A} follows along the same line.
\end{proof}

\begin{prop}
For an interval $I \in \R$, let $f : \Sigma \times I \to N^n$ and assume $\partial_t f = - \grad \mE (f)$ for all $t \in I$. Then it holds

\be \label{strev2ff}
\partial_t \AAA^f_{ij} = \D^2 \AAA^f_{ij} + P^2_3 + P^0_5 + Q^{2,1} + Q^{0,1}_{R*R}\,.
\ee
\end{prop}

\begin{proof}
We start by noticing that
\be \label{strV}
\grad \mE (f) = g^{ai}g^{bj} \grad^2_{ab} \AAA^f_{ij}  + P^0_3 + Q^{0,1}_0 \,,
\ee
From \eqref{var2ff} it follows that
\be \label{strevA}
\partial_t \AAA^f_{ij} =  \grad^2_{ij} \grad \mE (f) + P^0_2 * \grad \mE (f) + Q^{0,0}_0 * \grad \mE (f).
\ee
Putting together \eqref{com2A}, \eqref{com4A}, \eqref{strevA} and \eqref{strV} we get
\be
\begin{split}
\partial_t \AAA^f_{ij} &\overset{\eqref{strV}} = \grad^2_{ij} (g^{ap}g^{bq} \grad^2_{ab} \AAA^f_{pq}  + P^0_3 + Q^{0,1}_0) + P^0_2 * (g^{ap}g^{bq} \grad^2_{ab} \AAA^f_{pq}  + P^0_3 + Q^{0,1}_0)\\ 
                                     & \quad + Q^{0,0}_0 * (g^{ap}g^{bq} \grad^2_{ab} \AAA^f_{pq}  + P^0_3 + Q^{0,1}_0)\\
                                     & \overset{\eqref{com4A}} = g^{ap}g^{bq}\grad^2_{ab} \grad^2_{ij} \AAA^f_{pq} + P^2_3 + P^0_5 + \grad^2 Q^{0,1}_0 + Q^{0,3}_0 + Q^{2,1}_1 + Q^{0,1}_{R*R}\\
                                     & \overset{\eqref{com2A}} = \Delta^2 \AAA^f_{ij} + \grad^2 (P^0_3 + Q^{0,1}_1) + P^2_3 + P^0_5 + \grad^2 Q^{0,1}_0 + Q^{0,3}_0 + Q^{2,1}_1 + Q^{0,1}_{R*R}\\
                                     & = \Delta^2 \AAA^f_{ij} + P^2_3 + P^0_5 + Q^{2,1}_1 + Q^{0,1}_{R*R} \,. \\
\end{split}
\ee
\end{proof}
%
%
\section{Lifespan Theorem}
In this section we use some results proven in \cite{Link13} to obtain an estimate on the lifespan of a $\grad \mE-$flow in terms of the concentration of the $L^2-$norm of the second fundamental form of its initial datum.\\

\begin{Definition}
Let $f : \Sigma \times [0,T) \to (N^n,\bg)$ be a smooth one parameter family of isometric immersions of a  closed surface into a closed Riemannian manifold. We define the concentration of $\AAA^f$ at time $t$ and scale $\rho$ as
\be
\chi_f(\rho , t) := \sup_{x \in N^n} \int_{{f(\cdot , t)}^{-1} (\overline{B^{\bg}_{\rho}(x)})} |\AAA^{f(\cdot , t)}|^2 d \mu_{f(\cdot , t)} \,,
\ee
where $B^{\bg}_{\rho}(x)$ is the geodesic ball with centre at $p$ and radius $\rho$, with respect to the metric $\bg$.
\end{Definition}
\begin{Remark}
Equation \eqref{strev2ff} has the same structure as Equation (2.10) in \cite{Link13}. Thus, the following result, proved in \cite{KuSc02}  in the case $N^n = \R^n$ and in \cite{Link13} for an arbitrary closed ambient manifold, holds true also for  $\grad \mE-$flows.\\
\end{Remark}
\begin{thm} \label{th:ConcComp}
Given an isometric immersion $f_0 : (\Sigma, g) \to (N^n,\bg)$ of a closed surface into a closed Riemannian manifold, let $f: \Sigma \times [0,T) \to N^n$ be the maximal $\grad \mE-$flow with initial datum $f_0$.\\
For $\rho > 0$ and $\e > 0$, define
$$
t^+_{\e}(\rho) := \sup\{t \geq 0 : \chi(\rho,s) < \e^2, s \in [0,t) \} \,.
$$
Then there exists $\e_0((N^n, \bg)) >0$ such that either $T = t^+_{\e_0}(\rho) = \infty$, or there exist a constant $C$ for which
$$
T  > t^+_{\e_0}(\rho) \geq C \rho^4 \log \Bigg( \frac{C \e^2_0}{\chi(\rho,0) + \rho^4 ||\bgrad \bar{\RRR}||^2_{L^{\infty}(N,\bg)}(\mu_{f_0}(\Sigma) + \rho^2 \mW(f_0))} \Bigg)\,.
$$
\end{thm}
%
%
\section{Existence of the Blowup}
In this section we prove an existence result for blowups of $\grad \mE-$flows.\\
We start by stating a simplified version of a compactness theorem originally due to Langer, generalized by Breuning in \cite{Breu12} and extended by Cooper to the Riemannian setting, which will be used in the following arguments.
\begin{thm} \label{Cooper} \cite[Theorem 1.2]{Coop11}
Given a closed surface $\Sigma$, a sequence of closed Riemannian manifold  $\{(N^n,\bg_i)\}_{i \in \N}$ with uniformly bounded geometry and two sequences of points $\{p_i\}_{i \in \N} \subset \Sigma$, $\{x_i\}_{i \in \N} \subset N^n$, let $F_i : (\Sigma,g_i) \to (N^n,\bg_i)$ be a sequence of isometric proper immersions such that $F_i(p_i) = x_i$. Suppose that
\be \label{mbound}
\mu_{F_i}(B_R(x_i)) \leq C(R) \quad \textrm{for any} \quad R > 0\,,
\ee
\be \label{Abound}
||\grad^k \AAA^{F_i}||_{L^{\infty}}\leq C(k) \quad \textrm{for any} \quad  \textrm{and} \quad k \in \N \,,
\ee
where the covariant derivatives are taken with respect to the Levi--Civita connection associated to $g_i$.\\
Then there exist a surface $\hSig$, a complete Riemannian manifold $(M^n,\bg)$ and two points $p \in \hSig$ and $x \in M^n$ such that
\begin{itemize}
\item There exixts an increasing exhaustion $\{W_i\}_{i \in \N}$ of $\hSig$ made of open relatively compact sets, and there are diffeomorphisms $\phi_i : W_i \to \Sigma$ with $\phi_i(p) = p_i$, such that for any $R>0$ it holds $\Sigma^{g_i}_R(p_i) \subset \phi_i(W_i)$, for all $i \geq i_0(R)$
\item There is an increasing exhaustion $\{V_i\}_{i \in \N}$ of $M^n$ made of open relatively compact sets, and diffeomorphisms $\psi_i : V_i \to N^n$ with $\psi_i(x) = x_i$, such that for any $R>0$ it holds $B^{\bg_i}_R(x_i) \subset \psi_i(V_i)$, for all $i \geq i_0(R)$
\item $\psi^*_i \bg_i \to \bg$ smoothly
\item $\phi_i(W_i) \subset \psi_i(V_i)$
\item There exist a proper immersion $\hat{F}: \hSig \to M^n$ such that $\psi^{-1}_i \circ F_i \circ \phi_i \to \hat{F}$ smoothly with respect to a global isometric embedding of $(M^n,\bg)$ into a suitable Euclidean space $\R^K$. The immersion $\hF$ also satisfies \eqref{mbound} and \eqref{Abound} with respect to $x$.
\end{itemize}
\end{thm}

\begin{Remark}
From now on, a sequence of proper immersions $F_i : \Sigma \to N^n$ converging as in Theorem \ref{Cooper} to a proper immersion $\hat{F}: \hSig \to M^n$, will be denoted by short with $F_i \to  \hat{F}$. 
\end{Remark}

\begin{lemma} \label{semic_A2}
Under the hypothesis of Theorem \ref{Cooper} assume that it holds
\be \label{sch}
\int_{\Sigma} |\AAA^{F_i}|^2 d \mu_{F_i} \leq C < \infty \,.
\ee
Then it holds
\be \label{LowerSc}
\int_{\hSig}|\AAA^F|^2 d \mu_{\hF} \leq \liminf_{i \to \infty} \int_{\Sigma} |\AAA^{F_i}|^2 d \mu_{F_i}\,.
\ee
\end{lemma}
\begin{proof}
Since the surface $\hSig$ is in general not compact, the $L^2-$norm of its second fundamental form is given by
\be \label{LimFiniteA2}
\int_{\hSig}|\AAA^{\hF}|^2 d \mu_{\hF} = \lim_{R \to \infty} \int_{\hSig_R(x)} |\AAA^{\hF}|^2 d \mu_{\hF}.
\ee
We fix now $R>0$ and let $\e > 0$ be arbitrary. By \eqref{sch} and the locally smooth convergence of $F_i$ to $F$, we deduce that there exist an $i_{\e} \in \N$ such that for all $i > i_{\e}$ it holds
\be
\Big| \int_{\hSig_R(x)} |\AAA^{\hF}|^2 d \mu_{\hF} - \int_{\phi_i (\hSig_R(x))} |\AAA^{\psi_i^{-1}\circ F_i}|^2 d\mu_{\psi_i^{-1} \circ F_i} \Big| < \e \,,
\ee
(where the $\phi_i$ and the $\psi_i$ are the diffeomorphisms in the definition of the local convergence) and thus
\be
\begin{split}
\int_{\hSig_R} |\AAA^F|^2 d\mu_F & \leq  \int_{\phi_i (\hSig_R(x))} |\AAA^{\psi_i^{-1}\circ F_i}|^2 d\mu_{\psi_i^{-1} \circ F_i} + \e \leq \int_{\Sigma} |\AAA^{\psi_i^{-1}\circ F_i}|^2 d\mu_{\psi_i^{-1} \circ F_i} + \e \\
                                & \leq \int_{\Sigma} |\AAA^{F_i}|^2 d\mu_{F_i} + 2\e
\end{split}
\ee
for all $i > i_{\e}$. Taking the  liminf for $i \to \infty$ on the right hand side and then the  limit for $R \to \infty$ on the left hand side, by the arbitrariness of $\e$ we obtain \eqref{LowerSc}.
\end{proof}

The following is the main result of this section.
\begin{thm} \label{th:Blowup}
Let $f : \Sigma \times [0,T) \to (N^n,\bg)$ be a maximal $\grad \mE-$flow of a closed surface into a closed Riemannian Manifold. Suppose that
\be
\mu(f) := \sup_{t \in [0,T)} \mu_{f(\cdot,t)}(\Sigma) < \infty\,,
\ee
and that the flow concentrates at $T \in (0,\infty]$, which means
\be \label{eq:concentration}
\e^2_T := \lim_{\rho \to 0} (\limsup_{t \to T} \chi(\rho,t)) > 0\,.
\ee
Then there exist sequences $t_i \to T$, $r_i \to 0$ and $x_i \in N^n$ such that the rescaled flows
\be
f_i : (\Sigma , g_i) \times \Big[- \frac{t_i}{r^4_i} , \frac{T-t_i}{r^4_i} \Big) \to (N^n,\bg_i) \qquad f_i(p , t) := f(p , t_i + r^4_i t) \,,
\ee
with $\bg_i = r^{-2}_i \bg$ and $g_i = {f_i(\cdot , t)}^* \bg_i$, converge locally smoothly on $\hSig \times \R$ to a static $\grad \mE-$flow represented by a static properly immersed Willmore surface $\hf : \hSig \to \R^n$, open sets with the property
\be \label{eq:nontrivial}
\int_{\hf^{-1}(\overline{B_1(0)})} |\AAA^{\hf}|^2 d\mu_{\hf} > 0.
\ee
More precisely, there exist a sequence $h_i \to \infty$, local charts $\psi_i : B^{\bg_i}_{h_i}(x_i) \to V_i \supset B^n_{2i^2}$, a surface $\hSig$, open sets $U_i \subset \subset U_{i+1}$ such that $\cup^{\infty}_{i=1} U_i = \hSig$, diffeomorphisms $\phi_i : U_i \to f^{-1}_i(\cdot , 0)(\psi^{-1}_i(B^n_i))$, open sets $\Sig_i \subset \Sigma$ such that $\phi_i(U_i) \subset \Sig_i$ and $f_i(\Sig_i,J) \subset B^{g_i}_{h_i}(x_i)$ for each $J \subset \subset \R$, so that $\bar{f}_i \circ \phi_i \to \hf$ locally smoothly on $\hSig \times \R$ as maps to $\R^n$.
\end{thm}
\begin{Remark}
The condition $\mu(f) < \infty$ is always satisfied if $T < \infty$, since in this case it holds
\begin{equation*}
\mu_{f(\cdot,t)}(\Sigma) \leq C(f_0) \sqrt{t} \mE(f_0)^{1/2} + \mu_{f(\cdot,0)}(\Sigma) \,,
\end{equation*}
as it can be easily proven by means of the Cauchy-Schwarz inequality.\\
In the next section (Remark \ref{rmk:areabound}) we will make assumptions on the curvature tensor of the ambient manifold $(N^n,\bg)$ which ensure a uniform bound in time on $\mu_{f(\cdot,t)}(\Sigma)$ also in the case $T = \infty$.
\end{Remark}

The proof of Theorem \ref{th:Blowup} differs from the one in \cite{Link13} (Theorem 0.3) just in the last part, which we now prove.
\begin{lemma} \label{Staticflow}
The limit flow $\hf : \hSig \times \R \to \R^n$ is a static Willmore flow.
\end{lemma}
\begin{proof}
Let $\t_1, \t_2 \in \R$ with $\t_1 < \t_2$, $U \Subset \hSig$ be an open set, and $\phi_i$ the diffeomorphisms in the convergence statement of Theorem \ref{th:Blowup}. Then it holds
\begin{equation*}
\begin{split}
\int^{\t_2}_{\t_1} \int_U |\grad \mE(f_i \circ \phi_i)|^2 d\mu_{f(\phi (\cdot), t_i)} d \t & = \int^{\t_2}_{\t_1} \int_{\phi_i(U)} |\grad \mE(f_i)|^2 d\mu_{f(\cdot,t_i)} d \t \\
                                                                                                                                              & \leq \int^{\t_2}_{\t_1} \int_{\Sigma} |\grad \mE(f_i)|^2 d\mu_{f(\cdot,t_i)} d \t \\
                                                                                                                                              & = \mE(f_i)|_{\t=\t_1} - \mE(f_i)|_{\t=\t_2}\\
                                                                                                                                              & = \mE(f)|_{t_i + r^4_i \t_1} - \mE(f)|_{t_i + r^4_i \t_2}\,.
\end{split}
\end{equation*}
Since $t_i \to T$ and $r_i \to 0$, the locally smooth convergence $f_i \to \hf$ implies that $\grad \mE(\hf) = 0$. In view of the fact that $\mE (f)$ is uniformly bounded from above and lower semicontinuous with respect to the local smooth convergence of flows, we obtain that also $\mE(\hf)$ is uniformly bounded and we can apply the Gauss-Bonnet theorem (see \cite{MuSv95} for the case of non compact surfaces with finite total curvature) to obtain that on $\hf(\cdot,t)$ the  Willmore functional and the functional $\mE$ differ by a multiplicative factor (given by the normalization we have chosen) and by a constant depending only on the topology of the immersed surface. The thesis follows.
\end{proof}
\begin{Remark}
The limit surface $\hSig$ obtained in Theorem \ref{th:Blowup} could a priori have more than one connected component. In the following, we will always restrict our analysis to one of its connected components. Notice also that from \eqref{eq:nontrivial} it follows that the blowups we construct are non empty.
\end{Remark}
%
%
\section{Long time existence}
In this section we prove a long time existence theorem for $\grad \mE-$flows.\\

We preliminarly recall some results which will be used in the proof.
\begin{Lemma}\cite[Proposition 2.1]{KuMS14} \label{lem:GMT}
Let $f: \Sigma \times [0,T) \to N^n$ be a smooth one parameter family of immersions of a smooth closed surface $\Sigma$ into a smooth manifold $N^n$.
If $f$ satisfies
\be
E(f) := \sup_{t \in [0,T)} \int_{\Sigma} |\AAA^{f(\cdot,t)}|^2 d \mu_{f(\cdot,t)} <  \infty
\ee
and if the sectional curvatures $K^{N^n}$ of $N^n$ satisfy $\inf_{N^n} K^{N^n} > 0$, then it holds
\be
\mu(f) \leq \frac{1}{\inf_{N^n} K^{N^n}} (2 E(f) + 2 \pi \chi(\Sigma))\,.
\ee
\end{Lemma}
\begin{Remark}
The assumption on the positivity of the sectional curvatures will be needed to ensure that $\mu_{f(\cdot,t)}(\Sigma)$ stays uniformly bounded along a $\grad\mE-$flow satisfying $\mu_{f(\cdot , 0)}(\Sigma) < \infty$.
\end{Remark}

\begin{Lemma} \cite[Lemma 2.2]{KuMS14} \label{GMT0}
Let $F : \Sigma \to N^3$ be a closed immersion of a surface into a compact Riemannian manifold $(N^3,\bg)$. If $\mW(F) + \mu_F(\Sigma) \leq K < \infty$, there exist a constant constant $C > 0$ such that, for any $x \in N^3$ and $r>0$ holds
\be
\mu_F(\Sigma_r(x)) \leq C r^2 \,.
\ee
\end{Lemma}
\begin{Lemma} \label{Inversion} \cite[Lemma 4.1]{KuSc12} 
Let $F: \Sigma \to \R^n$ be an immersion of a surface such that $\int_{\Sigma} \HHH^2 d \mu_F < \infty$. Then there exists a point $x_0 \in \R^n$ and a ball $B_{\rho}(x_0)$ with centre at $x_0$ and radius $\rho > 0$ such that $F(\Sigma) \cap B_{\rho}(x_0) = \emptyset$.
\end{Lemma}
\begin{thm}\cite[Theorem 2]{White87}\label{th:White}
Let $F: \Sigma \to \R^3$ be an immersion of a connected oriented surface $\Sigma$, which is also complete with respect to the induced pullback Riemannian metric $g$. If $\int_{\Sigma} |\AAA^f|^2 d \mu_F < \infty$, then $\int_{\Sigma}k_g d\mu_F$ is an integral multiple of $4\pi$.
\end{thm}
The following interior estimates, as well as Theorem \ref{th:ConcComp}, depend just on the structure of equation \eqref{strev2ff} and have been first proven in \cite{Link13} for the Willmore flow in Riemannian manifolds. Thus they hold true for $\grad \mE-$flows as well.
\begin{thm} \cite[Lemma 3.3]{Link13}\label{th:InteriorEst}
Let $f : \Sigma \times [0,t] \to N^n$ be a $\grad \mE-$flow with $\mu(f) < \infty$. There exist constants $\rho_0 > 0$, $C((N^n,\bg)) >0$ and $\e_1((N^n,\bg)) > 0$ such that, if for a $\rho < \rho_0$, $t \leq C \rho^4$ we have
\begin{equation*}
\sup_{s \in [0,t]} \int_{f^{-1}(B^{\bg}_\rho(x),s)} |\AAA^{f(\cdot,s)}|^2 d\mu_{f(\cdot,s)} \leq \e_1\,,
\end{equation*}
then for every $k \in \N$ it holds
\be
||\grad^k \AAA||_{L^{\infty}(f^{-1}(B^{\bg}_{\rho/2}(x),s))} \leq c((N^n,\bg), k, C) s^{-\frac{k+1}{4}} \,,
\ee
for all $s \in (0,t]$.
\end{thm}
We now state and prove our main result.

\begin{thm} \label{th:Main}
Let $f: \SS^2 \times [0,T) \to N^3$ be a maximal $\grad \mE-$flow satisfying

\be \label{eq:BoundArea}
\mu(f) := \sup_{t \in [0,T)} \mu_{f(\cdot,t)} < \infty
\ee
and
\be \label{Ebound}
\mE(f(\cdot, 0)) \leq 8 \pi \,.
\ee
Then $T =  \infty$ and the flow do not concentrate.
\end{thm}

\begin{Remark} \label{rmk:areabound}
Assumption \eqref{eq:BoundArea} requires a control on the area of $f$ which is global in time. Nevertheless, if the sectional curvatures of $(N^3,\bg)$ are positive, the bound in \eqref{eq:BoundArea} is satisfied if the initial datum of the flow has finite area  (see Lemma \ref{lem:GMT}). As for condition \eqref{Ebound}, in \cite{KuMS14} (using results proven in \cite{Mond10} and \cite{Mond13}) it is shown that the existence of a point $x \in N^3$ at which the scalar curvature of $\bg$ is positive is sufficient to ensure that there exist immersions $F :\SS^2 \to N^3$ with $\mE(F) < 8 \pi$. 
\end{Remark}

\begin{proof}
If $\mE (f(\cdot , 0)) = 8 \pi$, then either $f(\cdot,0) : \SS^2 \to N^3$ is a critical immersion for $\mE$ and the theorem trivially holds, or $\mE(f(\cdot,t))$ is strictly monotone decreasing in time as long as the flow exists and we have
\be \label{eq:BoundE}
\mE(f(\cdot,t)) < 8 \pi \quad \textrm{for all} \quad t \in (0,T)\,.
\ee   
We assume by contraddiction that the flow concentrates at $T$.
With
\begin{equation*}
f_i (\SS^2 , g_i) \times \Big[- \frac{t_i}{r^4_i} , \frac{T-t_i}{r^4_i} \Big) \to (N^3,\bg_i)
\end{equation*}
we denote the sequence of flows constructed in Theorem \ref{th:Blowup} and we set $f^0_i (\cdot):= f_i(\cdot , 0)$. By the locally smooth convergence of flows $f_i \to \hf$, we get that the respective time slices locally smoothly converge and thus $f^0_i \to f^0$, where  $f^0: \Sigma \to \R^3$ is the immersion defined by $f^0(\cdot) = f(\cdot , 0)$. Notice that $f^0$ maps to $\R^3$, since $(N^3,\bg)$ is closed and along the sequence of dilations $(N^3,\bg_i := r^{-2}_i \bg)$ converges to $\R^3$ endowed with the Euclidean metric. Lemma \ref{semic_A2} implies
\be \label{eq:BoundLimA}
\int_{\Sigma}  |\AAA^{f^0}|^2 d \mu_{f^0} < 8 \pi
\ee
and by an elementary estimate, we obtain
\be \label{BoundLimH}
\int_{\Sigma} |\HHH^{f^0}|^2 d \mu_{f^0} = 4 \mW(f^0) < 16 \pi \,.
\ee
Since the minimum of the Willmore functional on compact immersions of closed surfaces into $\R^3$ is $4 \pi$, we conclude that $f^0(\Sigma)$ is not compact. Moreover, the bound in \eqref{eq:BoundLimA} allows to apply Corollary 4.3.2 in \cite{MuSv95} and we can conclude that the immersion $f^0$ is actually an embedding, as well as that $\Sigma$ is orientable.\\
The bound \eqref{eq:BoundLimA} and a result of Huber in \cite{Hub57}, ensure that $f^0(\Sigma)$ can be conformally parametrized over a compact Riemann surface $\hSig$, from which a finite number of points $\{p_1, ..., p_k\}$ have been removed.\\
We now prove that $\hSig$ is a sphere. To this aim we exploit the Gauss-Bonnet Theorem and \eqref{eq:BoundLimA} to obtain
\be
-4 \pi < - \frac{1}{2} \int_{\Sigma} |\AAA^{f^0}|^2 d \mu_{f^0} \leq \int_{\Sigma} k_g d \mu_{f^0} = 2 \pi (\chi(\hSig) - \sum^k_{j =1}(m_j+1)) \leq \frac{1}{2} \int_{\Sigma} |\AAA^{f^0}|^2 d \mu_{f^0} < 4\pi \,,
\ee
where $m_j$ is the multiplicity of the conformal immersion of $\Sigma \setminus \{p_1, ..., p_k\}$ in the neighborhood of the point $p_j$ (see \cite{MuSv95}). By Theorem \ref{th:White} we can conclude that $2\pi  (\chi(\hSig) - \sum^k_{p =1}(m_p+1))$ is a multiple of $4 \pi$ and this, together with the previous chain of inequalities, implies that
\be \label{chimult}
\chi(\hSig) = \sum^k_{p=1} (m_p+1) \,.
\ee
Being $\hSig$ a compact orientable surface, we have that $\chi(\hSig) \leq 2$, while on the other hand, by \eqref{chimult}, $\chi(\hSig) \geq 2$ holds. This means that $\chi(\hSig) = 2 = \sum^k_{p=1} (m_p+1)$. Consequently, we have that $p=1$ and hence $\Sigma$ has the topology of a two dimensional plane. Since Lemma \ref{Staticflow} ensures that $f^0 : \Sigma \to \R^3$ is a Willmore embedding, we can apply Lemma 4.1 in \cite{KuSc04} to deduce that the image of $f^0(\Sigma)$ under an inversion $J_{x_0} :\R^3 \to \R^3$ with respect to a point $x_0 \in  \R^3$ not belonging to $f^0(\Sigma)$ is a smooth Willmore surface (actually a sphere) with $\int_{\Sigma} |\HHH^{J_{x_0} \circ f^0}|^2 d \mu_{J_{x_0} \circ f^0} < 32 \pi$. The existence of such a point $x_0$ is a consequence of the finiteness of $\mW(f^0)$ and of Lemma \ref{Inversion}. By Bryant's classification of Willmore spheres (see \cite{Bry84}), the set $(J_{x_0} \circ f^0) (\Sigma)$ can be just a round sphere, since the value of the Willmore functional evaluated on other Willmore spheres would not satisfy the bound in \eqref{BoundLimH}. This means that $f^0(\Sigma)$ is a flat plane and we get a contraddiction with the non triviality of the blowup (ensured by \eqref{eq:nontrivial}). Thus $T = \infty$ holds and the flow do not concentrate.\\
\end{proof}

We now conclude by proving the subconvergence of the flow to a critical immersion. Following the line proposed in \cite{KuSc01}, we proceed by proving a uniform bound in time on the concentration of the second fundamental form along a $\grad \mE-$flow which satisfies the hypothesis of Theorem \ref{th:Main}.
\begin{Lemma} \label{lm:UnifBdA}
Under the hypotheses in Theorem \ref{th:Main}, there exists $r_0 > 0$ such that
\be \label{eq:UnifSTBoundA}
\int_{\SS^2_{r_0}(x)} |\AAA^{f(\cdot,t})|^2 d\mu_{f(\cdot,t)} < \e_1 \quad \textrm{for all} \,\, x \in N^3 \,\,\, \textrm{and} \,\,\,  t \in [0,\infty) \,,
\ee
where $\e_1$ is as in Theorem \ref{th:InteriorEst}.
\end{Lemma}
\begin{proof}
Suppose that thesis does not hold. Then there exist sequences $r_i \searrow 0$, $x_i \in N^3$ and $t_i \nearrow \infty$, such that $\int_{\SS^2_{r_i}(x_i)}|\AAA^{f(\cdot,t_i)}|^2 d\mu_{f(\cdot,t_i)} \geq \e_1$. This implies that the flow concentrates at $T=\infty$ (see equation \eqref{eq:concentration}), in contradiction with Theorem \ref{th:Main}.
\end{proof}
\begin{Proposition} \label{pr:Subconv}
Under the hypotheses in Theorem \ref{th:Main}, for any sequence $t_i \to \infty$, the sequence of immersions $f(\cdot ,t_i)$ converges (modulo subsequences) locally smoothly to an immersion which is critical for $\mE$.
\end{Proposition}
\begin{proof}
By Lemma \ref{lm:UnifBdA} and Theorem \ref{th:InteriorEst}, for any $t_i > 1$ we get
\be \label{eq:IntInf}
||\grad^k \AAA^{f(\cdot , t_i)}||_{L^{\infty}} \leq c(k) \,.
\ee
Lemma \ref{GMT0} gives
\be
\mu_{f({\cdot,t_i})}(B_R(x_i)) \leq C(R)\,.
\ee
The hypotheses of Theorem \ref{Cooper} are thus satisfied and we can deduce that there exists a proper immersion $\hf : \hSig \to N^3$ such that (modulo a subsequence) $f(\cdot , t_i) \to \hf$.\\
For $t \geq - t_i$, we consider the $\grad \mE-$flows $\hf_i (p,t) := f(p , t_i + t)$. These flows satisfy the bounds in \eqref{eq:IntInf} and their initial data converge to $\hf$. Thus, modulo a subsequence, $\hf_i \to \tilde{f}$, where $\tilde{f} : \hSig \times [0,\infty) \to N^3$ is a smooth $\grad \mE-$flow with initial datum $\hf$.\\
Estimating $\grad \mE (\hf)$ as in Lemma \ref{Staticflow}, yields $\grad \mE (\hf) = 0$ and the last claim is proven.
\end{proof}
\begin{Remark}
The analysis of the long time behaviour of $\grad \mE-$flows discussed in the present work can be adapted also to the case of the Willmore flow of surfaces into Riemannian manifolds. In particular, under conditions ensuring the uniform boundedness of the area of the evolving surfaces (see \cite{Mond14} for an analysis of some special cases) and guaranteeing the existence of immersions $F : \Sigma \to N^3$ with $\mW(F) < 4 \pi$, the analogues of Theorem \ref{th:Main} and Proposition \ref{pr:Subconv} can be proven.
\end{Remark}

\begin{ackn}
The author has been partially supported by the DFG Collaborative Research Center SFB/Transregio 71 and would like to warmly thank Prof. E. Kuwert for the many fruitful discussions and for his helpful comments on a preliminary version of the paper. 
\end{ackn}

\end{document}